\documentclass{article}
\usepackage[utf8]{inputenc}
\usepackage{amssymb,amsmath,amsthm,verbatim,xcolor}
\usepackage[all]{xy}
\usepackage{paralist}
\usepackage{caption}
\usepackage[normalem]{ulem} 

\def\Proj{\mathrm{Proj}}
\let\mathbb\mathbf
\def\bP{\mathbb{P}}

\def\cO{\mathcal{O}}

\def\c{\mathcal}
\def\wt{\widetilde}
\def\ol{\overline}
\def\L{\mathcal L}
\def\A{\mathcal{A}}

\def\P{\mathcal{P}}
\def\K{\mathcal K}
\def\cS{\mathcal{S}}
\def\X{\mathcal{X}}

\def\Hom{\mathrm{Hom}}
\def\Ext{\mathrm{Ext}}

\def\Cliff{\mathrm{Cliff}}

\newtheorem{thm}{Theorem}[section]
\newtheorem{prop}[thm]{Proposition}
\newtheorem{lemma}[thm]{Lemma}
\newtheorem{cor}[thm]{Corollary}
\newtheorem{definition}[thm]{Definition}

\theoremstyle{definition}
\newtheorem{example}[thm]{Example}
\newtheorem{rem}[thm]{Remark}

\newtheorem{parag}[thm]{} 
\theoremstyle{remark}
\newtheorem{notation}[thm]{Notation}

\def\O{\mathcal{O}}
\def\Z{\mathbf{Z}}
\def\Pic{\operatorname{Pic}}

\def\C{\mathbf{C}}

\newcommand{\Kprim}{\mathcal{K}^{\mathrm{prim}}}
\newcommand{\ie}{i.e.,\ } 
\newcommand{\restr}[2]{\left. #1 \right| _{#2}}
\newcommand{\trsp}[1]{\vphantom{#1}^{\mathsf T\!} #1}

\newcommand{\Q}{\mathbf{Q}}
\newcommand{\G}{\mathbf{G}}
\renewcommand\P{\mathbf{P}}
\let\bP\P

\newcommand{\bx}{\mathbf{x}}
\def\iota{i}
\newcommand{\sExt}{{\cal E}\kern -.1em xt} 

\newcommand{\wps}{weighted projective space}

\title{Deformations and extensions of Gorenstein weighted projective spaces}
\author{Thomas Dedieu, Edoardo Sernesi}
\date{}

\begin{document}

\maketitle

\centerline{\it Dedicated to Ciro Ciliberto on the occasion of his 70th birthday}

\begin{abstract}
We study the existence of deformations of all $14$ Gorenstein weighted
projective spaces $\P$ of dimension $3$ by computing the number of times
their general anticanonical divisors are extendable. In favorable
cases (8 out of 14), we find that $\P$ deforms to a $3$-dimensional
extension  of a general non-primitively polarized $K3$ surface.
On our way we show that each such $\P$ in its anticanonical model
satisfies property $N_2$, i.e., its homogeneous ideal is generated by
quadrics and the first syzygies are generated by linear syzygies,
and we compute the deformation space of the
cone over $\P$. This gives as a byproduct the exact number of times $\P$ is
extendable.
\end{abstract}

\section{Introduction}

Some topics related to this paper have been discussed and  worked out with our friend and colleague Ciro Ciliberto. It is a great pleasure for us to dedicate this work to him.

It is well known that there are precisely 14 Gorenstein weighted
projective spaces of dimension 3 (see \cite{yP04}; we give the list in
Table~\ref{tab:Gwps}).
In this paper we introduce a method in the study of their
deformations,  consisting in studying simultaneously the deformations
and the extendability of their general anticanonical divisors. The
underlying philosophy  goes back to Pinkham \cite{hP74}, and then Wahl \cite{jW97} who showed   the
close connection between the existence of extensions of a projective
variety $X\subset \bP^r$ and the deformation theory of its affine cone
$CX\subset \mathbf{A}^{r+1}$. We discuss and recall this connection in
\S \ref{S:extendability}.  Wahl's interest was focused  on   canonical
curves, aiming  at a characterization of  those curves which are
hyperplane sections of a $K3$ surface (``$K3$ curves'') by means of
the behaviour of their gaussian map, thereafter called ``Wahl map''.
His program was carried out in \cite{ABS17} and applied in
\cite{CDS20} to the extendability of canonical curves, $K3$ surfaces
and Fano varieties.
The present work follows the same direction, as surface (resp.\ curve)
linear sections of Gorenstein weighted projective $3$-spaces in their
anticanonical embeddings are $K3$ surfaces with at worst canonical
singularities (resp.\ canonical curves).

Our starting point is the observation that  most
polarized general anticanonical divisors $(S,-\restr {K_\P} S)$ of the
\wps{s}  $\bP$
of Table~\ref{tab:Gwps} are non-primitively polarized (and singular).  This
suggests  considering a 1-parameter smoothing $(\cS,\L)$ of
$(S,-\restr {K_\P} S)$ and exploiting the fact that the extendability of
non-primitively polarized $K3$ surfaces is well understood, thanks to
work of Ciliberto--Lopez--Miranda \cite{CLM98},
Knutsen \cite{aK20},
and Ciliberto--Dedieu \cite{CDdouble, CDhigher}
(see \S \ref{S:nonprimitive} for details).
This plan works fine when the extendability
of $(S,-\restr {K_\P} S)$ coincides with that 
of the general non-primitively polarized $K3$ surface,
\ie the invariant
\[
  \alpha(S_t,L_t) = h^0(S_t,N_{S_t/\P^g}\otimes L_t^{-1})-g-1
\]
introduced in \ref{T:lvovski} takes the same value for all fibres of
$(\cS,\L)$.
What we get in this case is that $\P$ deforms to a threefold extension
of the general member of $(\cS,\L)$.
The final output (see Section~\ref{S:examples}) is an understanding of the deformation properties of
Cases $\#i$,
$i \in \{ 1, \dots, 7, 9\}$, in the notation of Tables~\ref{tab:Gwps}
and \ref{tab:alpha}.
A finer analysis is required for the other cases, which we don't try
to carry out here, although we make a couple of observations 
at the end of Section~\ref{S:examples}.

Our strategy involves various substantial technical verifications.
The main point is
controlling the deformation theory of the affine and projective cones
over possibly singular $K3$ surfaces. In the nonsingular case this is
a well known chapter of deformation theory, due to Schlessinger
\cite{mS73}: we extend it to the singular case in \S
\ref{S:extendability}. It is a non-trivial task to compute the
relevant deformation spaces in our examples,  and for this purpose we
took advantage of the computational power of Macaulay2 \cite{M2}.
Still this leaves some obstacles to the human user
(see the proof of Proposition~\ref{p:alpha-val} and the comments
thereafter), which we have
found are best coped with by considering more generally deformations
of cones over arithmetically Cohen--Macaulay surfaces or
arithmetically Gorenstein curves.

As a further reward of this computation, we obtain the exact number of
times each Gorenstein \wps{} of dimension $3$ is extendable
(Corollary~\ref{c:extend} and Table~\ref{tab:alpha}). We have not been
able however to identify the maximal extension in all cases.

Another condition we had to verify in order to apply
Wahl's criterion (Theorem~\ref{T:wahl}) is that the projective
schemes $X$ involved satisfy  condition $N_2$ (Definition~\ref{def:N2}),
so that ``each first order ribbon over $X$ is integrable to at most one
extension of $X$''.
We carry this out again with the computer and Macaulay2
(see Proposition~\ref{p:N2}), by explicitly computing the homogeneous
ideals of all Gorenstein weighted projective spaces in their
anticanonical embeddings, as well as their first syzygy modules.

Some of our end results about Gorenstein Fano threefolds in
Section~\ref{S:examples} can also be
obtained by direct calculations, using
computational tricks on weighted projective spaces similar to those
employed by Hacking in \cite[\S 11]{Hacking}, and showed to us by the
referee. We find it nice that the observations we made indirectly
using deformation theory may be confirmed by direct computations of a
different nature. 
Let us also mention the article \cite{Manetti} (which has been
continued in \cite{Hacking}),
in which degenerations
of the projective plane to various weighted projective planes are
exhibited: this is similar in spirit to what we do in
Section~\ref{S:examples}. 

The organization of the article is as follows.
In \S \ref{S:wps} we gather some elementary facts about weighted
projective geometry, and give the list of all $14$ Gorenstein weighted
projective spaces of dimension $3$.
In \S \ref{S:nonprimitive} we give a synthetic account of the
extension theory of non-primitive polarized $K3$ surfaces along the
lines of \cite{CDS20} and taking advantage of \cite{aK20}.
Section~\ref{S:extendability} is the technical heart of the paper and
is devoted to the deformation theory of cones and its application to
extensions.
This leads to our main technical result Theorem~\ref{T:smoothing} in
the following \S \ref{S:thearg}.
In Section~\ref{S:comput} we carry out the explicit computations
required for our application of Theorem~\ref{T:smoothing} to
Gorenstein \wps{s},
and in the final \S \ref{S:examples} we give the explicit output of
this application.

\medskip\noindent
\emph{Acknowledgements.}  ES thanks Alessio Corti and Massimiliano Mella for  enlightening
conversations.
ThD thanks Laurent Busé for having shown him the elimination
technique enabling the computation of the ideal of a weighted
projective space, and Enrico Fatighenti for introducing him to the
Macaulay2 package ``VersalDeformations''.
We thank the referee for useful comments and bibliographical
suggestions.


\section{Gorenstein weighted projective spaces}
\label{S:wps}

We will consider some weighted projective spaces (wps for short) of
dimension 3. In this section we collect some preliminary
definitions and basic facts. The authoritative reference is
\cite{iD82}; we will also rely on \cite{BR86} and \cite{fletcher}.

\begin{parag}\label{p:wps}
Consider a weighted projective 3-space of the form
$\bP:=\bP(a_0,a_1,a_2,a_3)$, where the $a_i$'s are relatively prime 
positive integers.

It is not restrictive to   further assume, and we will do it,  that
any three of the $a_i$'s are relatively prime, in which case one says
that $\bP$ is \emph{well formed}. 
Let:
$$
m := \mathrm{lcm}(a_0,a_1,a_2,a_3), \quad  s:= a_0+a_1+a_2+a_3.
$$
The following holds:

\begin{itemize}
    \item[(1)] For all $d \in \Z$ the sheaf $\cO(d)$ is reflexive of
      rank 1, and it is invertible if and only if $d=km$  for some $k \in \Z$ \cite[\S 4]{BR86}.
    
    \item[(2)]  Pic$(\bP) = \Z \cdot[\cO(m)]$
      \cite[Thm. 7.1, p.~152]{BR86}.
    
    \item[(3)] $\bP$ is Cohen--Macaulay and its dualizing sheaf is
      $\omega_\bP =
      \cO(-s)$. Therefore $\bP$ is Gorenstein if and only if $m|s$
      \cite[Corollary 6B.10, p.~151]{BR86}. In this case $\bP$ has
      canonical singularities because it is a Gorenstein
      orbifold. This  follows for  example from
      \cite[Prop.~1.7]{mR80}. 
    
    \item[(4)] The intersection product in $\bP$ is determined by
      (see \cite[p.~240]{jK99}):
    $$ \cO(1)^3 = \frac{1}{a_0a_1a_2a_3}.
    $$
\end{itemize}
\end{parag}

\begin{lemma}\label{L:thomas}
Let $S\subset \bP(a_0,a_1,a_2,a_3)$ be a general hypersurface of
degree $d$, such that $\O(d)$ is locally free on $\P$.
For all $k \in \Z$,
the restriction to $S$ of $\cO(k)$ is locally free if and only if 
$$
\forall i\ne j:
\quad 
\mathrm{gcd}(a_i,a_j)|k.
$$
\end{lemma}

\begin{proof}
The local freeness needs only to be checked at the singular points.
As $S$ is general it may be singular only along the singular locus of
$\P$, hence only along the lines joining two coordinate points
$P_i= (0:\ldots:1:\ldots:0)$.
Moreover $S$ avoids all coordinate points themselves thanks to our assumption on the degree.

Let $P$ be a point on the line $P_iP_j$, off $P_i$ and $P_j$.
In the local ring
$\O_P$ we have the invertible monomials $x_i^{n_i}x_j^{n_j}$ whose
degrees sweep out 
\[
a_i\Z + a_j\Z = \gcd(a_i,a_j)\Z.
\]
This shows that if $\gcd(a_i,a_j)$ divides $k$ then $\O(k)$ is
invertible at all points of $P_iP_j$ but $P_i$ and $P_j$ themselves,
hence the 'if' part of the statement.
The 'only if' part follows in the same way.
\end{proof}

\begin{parag}\label{p:Gwps}
It follows from \ref{p:wps}  that there are exactly 14 distinct
$3$-dimensional weighted projective spaces which are Gorenstein,
see \cite{yP04}.
We list them in Table~\ref{tab:Gwps} below, together with the following information. 
For each $\P$ in the list, we denote by $S$ a general anticanonical
divisor:
it is a $K3$ surface with $ADE$ singularities
\cite{mR83}.
We also use the following notation:
\begin{compactitem}[--]
\item $m$ is the lcm of the weights, so that $\O(m)$ generates
  $\Pic(\P)$; 
\item $s$ is the sum of the weights, so that
  $\omega_\P = \O(-s)$;
\item $\iota_S$ denotes the divisibility of $\restr {K_{\bP}} S$ in
  $\Pic(S)$, which is readily computed with Lemma~\ref{L:thomas}
  above;
\item $g_1$ is the genus of the primitively polarized $K3$ surface
  $(S,L_1)$, where $L_1 = - \tfrac 1 {\iota_S} \restr{K_\P} S$;
  we reserve the symbol $g$ to the common genus of 
 the Fano variety $\P$ and
the polarized $K3$ surface $(S,- \restr {K_{\bP}} S)$, \ie $2g-2=-K_\bP^3$.
\end{compactitem}
The rows are ordered according to $g_1$, then $\iota_S$ (decreasing),
then the weights.
We also indicate the singularities of $S$, which may be found
following \cite{fletcher}.

\begin{center}
\begin{tabular}{|l|l|l|l|l|l|l|l|}
  \hline
  \# & weights        & \raisebox{0mm}[3.5mm]{$-K_\P^3$} & $m$ & $s$ & $\iota_S$ & $g_1$ &
$\mathrm{Sings}(S)$
  \\
  \hline
 1 & (1, 1, 1, 3)   &   72 &      3 &  6 &     6 &   2 & smooth                 \\
 2 & (1, 1, 4, 6)   &   72 &     12 & 12 &     6 &   2 & $A_1$                  \\
 3 & (1, 2, 2, 5)   &   50 &     10 & 10 &     5 &   2 & $5A_1$                 \\
 4 & (1, 1, 1, 1)   &   64 &      1 &  4 &     4 &   3 & smooth                 \\
 5 & (1, 1, 2, 4)   &   64 &      4 &  8 &     4 &   3 & $2 A_1$                \\
 6 & (1, 3, 4, 4)   &   36 &     12 & 12 &     3 &   3 & $3A_3$                 \\
 7 & (1, 1, 2, 2)   &   54 &      2 &  6 &     3 &   4 & $3A_1$                 \\
 8 & (1, 2, 6, 9)   &   54 &     18 & 18 &     3 &   4 & $3A_1$, $A_2$          \\
 9 & (2, 3, 3, 4)   &   24 &     12 & 12 &     2 &   4 & $3A_1$, $4A_2$         \\
10 & (1, 4, 5, 10)  &   40 &     20 & 20 &     2 &   6 & $A_1$, $2A_4$          \\
11 & (1, 2, 3, 6)   &   48 &      6 & 12 &     2 &   7 & $2A_1$, $2A_2$         \\
12 & (1, 3, 8, 12)  &   48 &     24 & 24 &     2 &   7 & $2A_2$, $A_3$          \\
13 & (2, 3, 10, 15) &   30 &     30 & 30 &     1 &  16 & $3A_1$, $2A_2$, $A_4$  \\
14 & (1, 6, 14, 21) &   42 &     42 & 42 &     1 &  22 & $A_1$, $A_2$ , $A_6$ \ \\
\hline
\end{tabular}
\captionof{table}{Gorenstein $3$-dimensional weighted projective
  spaces}
\label{tab:Gwps}
\end{center}

\end{parag}

\section{Extendability  of  non-primitive polarized $K3$ surfaces}\label{S:nonprimitive}

Let us first recall the following.

\begin{definition}
A projective variety $X\subset \bP^r$ of dimension $d$ is
called \emph{$n$-extendable} for some $n \ge 1$ if there exists a
projective variety $\wt X\subset \bP^{r+n}$ of dimension $d+n$, not a
cone, such that  $X=\wt X \cap \bP^r$ for some linear embedding
$\bP^r\subset \bP^{r+n}$. The variety $\wt X$ is called an
\emph{$n$-extension} of $X$. If $n=1$ we call $X$ \emph{extendable}
and  $\wt X$  an \emph{extension} of $Y$.
\end{definition}

We refer to \cite{survey-angelo} for a beautiful tour on this subject.

 If $X$ is a Fano 3-fold of genus $g=-\frac{1}{2}K_X^3+1$, then a
general $S\in |-K_X|$ is a $K3$ surface naturally endowed with the ample
divisor $-\restr {K_{X}} S$ which makes $(S,-\restr {K_{X}} S)$ an
extendable polarized 
$K3$ surface of genus $g$. Suppose that $X$ has index $\iota_X > 1$,
\ie $-K_X=\iota_XH$ for an ample divisor $H$, indivisible in $\Pic(X)$.  Then
$(S,-\restr {K_{X}} S)$
is non-primitive because $-\restr {K_{X}} S=\iota_X \restr H S$ is
at least 
$\iota_X$-divisible.  Therefore by considering Fano 3-folds of index
$>1$ we naturally land in the world of extendable non-primitively
polarized $K3$ surfaces.

\begin{notation}
We denote by
$\K_g^k$ the moduli stack of polarized $K3$ surfaces of
  genus $g$ and index $k$, 
\ie pairs $(S,L)$ such that $S$ is a   $K3$ surface, possibly with ADE singularities,  and $L$ is
an ample and  globally generated
line bundle on $S$ with $L^ 2=2g-2$, such that $L=kL_1$ with $L_1$ a
primitive line bundle on $S$;
note that $(S,L_1)$ belongs to $\K_{g_1} ^1$, which we usually
denote by $\Kprim _{g_1}$, where
$2g_1-2 = L_1^2$ and $g=1+k^2(g_1-1)$.
\end{notation}

We have the following necessary condition for the extendability of a projective variety:

\begin{thm}[\cite{sL92}]\label{T:lvovski}%
Let $X\subset \P^ n$ be a smooth, projective, irreducible,
non-degenerate variety, not a quadric,
and write $L=\O_X(1)$. Set 
\[
\alpha(X,L)=h^ 0(N_{X/\P^n} \otimes L^{-1})-n-1.
\]
If $\alpha(X,L)<n$, 
then $X$ is at most $\alpha(X,L)$-extendable.
\end{thm}

When the polarization of $X$ is clear from the context, we write
$\alpha(X)$ instead of $\alpha(X,L)$.
Note that if $X$ is a smooth $K3$ surface or Fano variety (resp.\ a
canonical curve, hence $L=K_X$) then  
$$
\alpha(X)=  H^1(X,T_X\otimes L^{-1})
\quad
(\mathrm{resp.}\ \mathrm{cork}(\Phi_{\omega_X}),
$$ 
with $\Phi_{\omega_X}$ the Gauss--Wahl map of $X$,
see for instance \cite[\S 3]{CDS20}.

For $K3$ surfaces and canonical curves, the converse to
Theorem~\ref{T:lvovski} also 
holds, under some conditions. Precisely we have:

\begin{thm}[\cite{CDS20}, Thm.~2.1 and Thm.~2.17]\label{T:CDS}
Let $(X,L)$ be a smooth polarized $K3$ surface
(resp.\ $(X,L)=(C,K_C)$ a canonical curve) of genus $g$.
Assume that
$g \ge 11$ and $\mathrm{Cliff}(X,L) \ge 3$.
Then $(X,L)$ is $\alpha(X,L)$-extendable.

More precisely, every non-zero
$e\in H^1(X,T_X\otimes L^{-1})$
(resp.\ $e \in \ker (\trsp {\Phi_{\omega_C}} )$)
defines an  extension $X_e$ of $X$
which is unique up to projective automorphisms of $\bP^{g+1}$
(resp.\ $\bP^g$) fixing every point of $\bP^g$ (resp.\ $\bP^{g-1}$),
and there exists a \emph{universal extension}
$\tilde X \subset \P^{g+\alpha(X,L)}$ (resp.\ $\P^{g-1+\alpha(X,L)}$)
of $X$ having each $X_e$ as a linear section containing $X$.
\end{thm}

When $(X,L)$ is a smooth polarized $K3$ surface, by definition,
$\Cliff (X,L)$ is the \emph{Clifford index} of any   nonsingular curve  $C\in |L|$; by \cite{mR76b,DM89,GL87},   this does not depend on the choice of $C$. 
Note that in case $(X,L)$ is a $K3$ surface the extension $X_e$ in the theorem is an arithmetically
Gorenstein Fano variety of dimension three  with canonical
singularities.

 Unfortunately $H^1(X,T_X\otimes L^{-1})$ is not easy to compute in general, but in the non-primitive setting we can reduce to a more amenable case.

\begin{lemma}\label{L:CLM98}
Let $S\subset \bP^g$ be a smooth $K3$ surface,
and $L = \restr {\O_{\P^g}(1)} S$. Then:
$$
H^1(S,T_S\otimes L^{-j})=
\begin{cases}
  \mathrm{coker}
  \bigl[
  H^0(S,L)^\vee \to   H^0(S,N_{S/\bP^g}(-1))
  \bigr],
  &\text{if $j=1$} \\
  H^0(S,N_{S/\bP^g}\otimes L^{-j}),
  & \text{if $j\ge 2$}.
\end{cases}
$$
\end{lemma}

\begin{proof}
See  \cite{CLM98} (2.8).
\end{proof}

This lemma, applied to a smooth $(S,L_1)\in \K_{g_1}$ with $L_1$ very ample,
tells us that $(S,L)=(S,jL_1)$ with $j\ge 2$
is extendable
if and only if $H^0(S,N_{S/\bP^{g_1}}\otimes L_1^{-j})\ne 0$.  The
possibilities for the pair $(S,L_1)$ and $j$ are then very limited.
In particular, $H^0(S,N_{S/\bP^{g_1}}\otimes L_1^{-j})= 0$ for all $j \geq 2$
as soon as $S$ satisfies the property $N_2$ (see
Definition~\ref{def:N2}), see \cite[Lem.~1.1]{aK20} and the references therein.
In fact the possibilities have been completely classified
by A.~L.~Knutsen \cite{aK20},
see also \cite{CLM98} and \cite{CDhigher}.
The result is the
following:

\begin{thm}[\cite{aK20}]
\label{T:extveryample1}
Let $(S,L_1)\in \Kprim_{g_1}$ with $S$ smooth and $L_1$ very ample.
Then
$H^1(S,T_S\otimes L_1^{-j})=0$
for all $j\ge 2$ except in the following cases:

\begin{center}
\normalfont
\begin{tabular}{|c|c|c|c|l|}
\hline 
$g_1$&$j$&$g(L_1^{j})$&$h^1(S,T_S\otimes L_1^{-j})$& Notes \\
\hline 
   3&2&9&10& any $(S,L_1)$ \\
   3&3&19&4& any $(S,L_1)$ \\
   3&4&33&1& any $(S,L_1)$ \\
    4&2&13&5& any $(S,L_1)$ \\
    4&3&28&1& any $(S,L_1)$ \\
    5&2&17&3& any $(S,L_1)$ \\
    6&2&21&1& any $(S,L_1)$ \\
    7&2&25&1&(1) \\
    8&2&29&1&(2)\\
    9&2&33&1&(3)\\
    10&2&37&1&(4)\\
    \hline
\end{tabular}
\end{center}
where
\begin{compactenum}[\normalfont (1)]

    \item  $S$ is one of the following:

      \begin{compactenum}[\normalfont (I)]

      \item     a divisor in the linear system $|3H-3F|$ on the
        quintic  rational normal scroll $T\subset \bP^7$ of type
        $(3,1,1)$, with $H$ a hyperplane section and $F$  a fibre of
        the scroll.

      \item   a quadratic section of the sextic Del Pezzo threefold $\bP^1\times\bP^1\times\bP^1\subset \bP^7$ embedded by Segre.
    
      \item   the section of $\bP^2\times \bP^2\subset \bP^8$, embedded by Segre, with a hyperplane and a quadric.
    
  \end{compactenum}
  
\item   $S$ is an anticanonical divisor in a septic Del Pezzo 3-fold (the blow-up of $\bP^3$ at a point). 

\item  $S$ is    one of the following:

      \begin{compactenum}[\normalfont (I)]

      \item  the 2-Veronese embedding of a quartic of $\bP^3$; equivalently a quadratic section of the Veronese variety $v_2(\bP^3)\subset\bP^9$.

      \item  a quadratic section of the cone over the anticanonical embedding of the Hirzebruch surface $\mathbb{F}_1\subset\bP^8$.
   
      \end{compactenum}
      
\item   $S$ is a quadratic section of the cone over the Veronese
  surface $v_3(\bP^2)\subset \bP^9$.   
\end{compactenum}
In all cases, except $g_1=3$ and $j=2$, we have $\mathrm{Cliff}(S,L^j)\ge 3$.
\end{thm}

\begin{proof}
Using the identification given by Lemma \ref{L:CLM98}, the proof reduces to list all possible cases described by Proposition 1.4 of \cite{aK20}.  The final statement is an easy calculation.
\end{proof}

The case $g_1=3$ and $j=2$ is not liable to Theorem~\ref{T:CDS} as the
Clifford index in this case is too small, but it
has been studied by hand in \cite{CDdouble,CDhigher}.

Theorem \ref{T:extveryample1} does not cover the cases of $(S,L_1)$
hyperelliptic.  We shall only consider the case $g(L_1)=2$, which will be
sufficient for our purposes.
The following result has been obtained in \cite{CDdouble} by geometric
means; we give here a cohomological proof.

\begin{lemma}\label{L:genus2}
Let $(S,L_1)\in \Kprim_2$. Then the dimension of
$H^1(S,T_S\otimes L_1^{-j})$ takes the values given by the following table:
\begin{center}
\normalfont
\begin{tabular}{|c|c|c|c|}
\hline
  $j$& \raisebox{0mm}[4mm]{$g(L_1^j)$} & $\Cliff(L_1^j)$ &
$h^1(S,T_S\otimes L_1^{- j})$ \\
  \hline
1&2&0&18 \\
2&5&0&15 \\
3&10&2&10 \\
4&17&$>2$&6\\
5&26&$>2$&3\\
6&37&$>2$&1\\
$\ge 7$ &&  $>2$ &0\\
\hline
 \end{tabular}
\end{center}
\end{lemma}

   Recalling Theorem \ref{T:CDS} this table implies that $(S,L^j_1)$ is extendable for $4\le j \le 6$, since  $\Cliff (S,L^j_1)\ge 3$. In particular  $(S,L^6_1)\in \K_{37}$ and  is precisely 1-extendable; in fact it is hyperplane section of $\bP(1,1,1,3)$.
Lemma \ref{L:genus2} also tells us that the surfaces $(S,L^4_1)$ and
$(S,L^5_1)$ are  6-extendable and 3-extendable respectively, in
agreement with the results in  \cite{CDdouble}, (4.8). There, also the
situation in case $j=3$ (to which Theorem~\ref{T:CDS} does not apply)
is completely described.

\begin{proof}
The elementary computation of $\Cliff(L_1^{j})$ is left to the
reader. The surface $S$ is a double plane $\pi:S\longrightarrow \bP^2$
branched along a sextic $\Gamma$ and $L_1=\pi^*\cO_{\bP^2}(1)$.
Denote by $R$ the ramification curve of $\pi$. We have
$\cO_S(R)=L_1^{3}$. The cotangent sequence of $\pi$ is
$$
0 \to \pi^*\Omega^1_{\bP^2}\longrightarrow \Omega^1_S \longrightarrow \Omega^1_{S/\bP^2}\to 0
$$
where $\Omega^1_{S/\bP^2}=\cO_R(-R)
=L_1^{-3}\otimes\cO_R=\omega_R^{-1}$.
Therefore for every $j$ we have  the following diagram, where the vertical sequence is the twisted Euler sequence restricted to $S$:
$$
\xymatrix{
&0\ar[d]\\
0\ar[r]
&\pi^*\Omega^1_{\bP^2}\otimes L_1^{ j} \ar[d]\ar[r]
&\Omega^1_S\otimes L_1^{ j} \ar[r]
&L_1^{ j-3} \otimes\cO_R\to 0\\
&L_1^{ j-1}\otimes H^0(S,L_1) \ar[d]\\
&L_1^{ j}\ar[d]\\
&0
}
$$
For $j \ge 7$ this diagram gives
$H^1(S,\Omega^1_S\otimes L_1^{ j})=0$.
If $j=1$ we get the following exact sequence:
$$
\xymatrix{
0\to H^1(S,\Omega^1_S\otimes L_1)\ar[r]
&H^1(R,L_1^{-2}\otimes\cO_R)\ar[r]
&H^2(S,\pi^*\Omega^1_{\bP^2}\otimes L_1) \ar@{=}[d]\ar[r]
&0\\
&&H^2(S,\cO_S)\otimes H^0(S,L_1)
}
$$
which gives $h^1(S,\Omega^1_S\otimes L_1)
=h^1(R,L_1^{-2}\otimes\cO_R)-3=18$.

If $j=2$ we have 
$$
h^1(S,\pi^*\Omega^1_{\bP^2}\otimes L_1^{ 2})
= \mathrm{corank} \bigl[ \mathrm{Sym}^2H^0(S,L_1)
\to H^0(S,L_1^{2}) \bigr]=0
$$
and the following exact sequence:
$$
0\to H^1(S,\pi^*\Omega^1_{\bP^2}\otimes L_1^{2})
\longrightarrow H^1(S,\Omega^1_S\otimes L_1^{2})
\longrightarrow H^1(R,L_1^{-1}\otimes\cO_R)\to 0
$$
which gives $h^1(S,\Omega^1_S\otimes L_1^{2})
= 0+h^1(R,L_1^{-1}\otimes\cO_R)=15$.

If $3\le j \le 6$ then
$$h^1(S,\pi^*\Omega^1_{\bP^2}\otimes
L_1^{j})=0$$
thus $H^1(S,\Omega^1_S\otimes L_1^{j})\cong
H^1(R,L_1^{j-3}\otimes\cO_R)$, and the conclusion is clear.
\end{proof}

\section{Extendability and graded deformations of cones}\label{S:extendability}

Consider a projective scheme $X\subset \bP^r$   and let
$A=R/I_X$ be its homogeneous coordinate ring,
where $R=\mathbb{C}[X_0,\dots,X_r]$
and $I_X$ is the saturated homogeneous ideal of $X$ in $\P^r$.
The  \emph{ affine cone} over $X$ is
$$
 CX:=\mathrm{Spec}(A)\subset \mathbb{A}^{r+1}
 $$
 and the \emph{projective cone} over $X$ is 
 $$\ol{CX}:=\mathrm{Proj}(A[t])\subset \bP^{r+1}.
 $$

 Recall the following standard definitions. The scheme $X$ is \emph{projectively
   normal}, resp.\ \emph{arithmetically Cohen-Macaulay},
 resp.\ \emph{arithmetically Gorenstein} if the local ring of  $CX$ at
 the vertex is integrally closed, resp.\ Cohen-Macaulay, resp.\ Gorenstein.   Also recall that if   $X$ is  normal and  arithmetically Cohen-Macaulay then it is projectively normal.

   The deformation theory of $CX$ is controlled by the cotangent modules $T^1_{CX}$ and $T^2_{CX}$, which are graded 
    because of the $\mathbb{C}^*$-action on $A$. We will only need
    the explicit description of the first one.

\begin{prop}\label{P:schless}
Let $X\subset \bP^r$  be a non-degenerate 
scheme of pure dimension $d\ge 1$. Consider  the following conditions:
\begin{compactenum}[\normalfont (a)]
\item\label{cond:aCM}
  $X$ is arithmetically Cohen--Macaulay (aCM for short).

\item\label{cond:projnorm}
  $X$ is projectively normal.

\end{compactenum}
If either \eqref{cond:aCM} or \eqref{cond:projnorm}
holds then we have an exact sequence
of graded modules:
    \begin{equation}\label{E:T12}
      \bigoplus_{k\in \mathbb{Z}}  H^0(X,\restr {T_{\bP^r}} {X}  (k))
      \longrightarrow
  \bigoplus_{k\in \mathbb{Z}} H^0(X,N_{X/\bP^r}(k))
  \longrightarrow  T^1_{CX}\to 0.
\end{equation}
\end{prop}

   \begin{proof}
Let $v\in CX$ be the vertex, $W=CX\setminus \{v\}$ and let
$\pi:W\longrightarrow X$ be the projection. By definition we have an
exact  sequence:
  \begin{equation}\label{E:T1exseq}
     H^0(CX,\restr {T_{\mathbb{A}^{r+1}}} {CX})\longrightarrow H^0(CX,N_{CX/\mathbb{A}^{r+1}}) \longrightarrow T^1_{CX}\to 0.
  \end{equation}
We assume that \eqref{cond:aCM} or
\eqref{cond:projnorm} holds.
Then $CX$ verifies Serre's condition $S_2$ at the vertex.
The two sheaves $F$ respectively involved in the two first terms of
\eqref{E:T1exseq} are reflexive,
each being the dual of a coherent sheaf,
hence they have  depth $\ge 2$ at $v$ as well by
\cite[Prop.~1.3]{HartSRS}
(for the implication we use, it is enough that the $X$ from the
notation of ibid.\ be $S_2$, as the proof given there shows).
Therefore
$$
H^0(CX,F)\cong H^0(W,\restr F {W}).
$$
Thus  \eqref{E:T1exseq} induces an exact sequence
$$
H^0(W,\restr {T_{\mathbb{A}^{r+1}}} {W})\longrightarrow H^0(W,N_{W/\mathbb{A}^{r+1}}) \longrightarrow T^1_{CX}\to 0.
$$
As in the proof of \cite[Lemma 1]{mS73}, one sees that
$$
H^0(W, \restr {T_{\mathbb{A}^{r+1}}} {W}  )=
\bigoplus_{k\in \mathbb{Z}} H^0(X,\cO_X(k+1))^{r+1}
$$
and
$$
H^0(W,N_{W/\mathbb{A}^{r+1}})=
\bigoplus_{k\in \mathbb{Z}} H^0(X,N_{X/\bP^r}(k)).
$$
Then we have a commutative diagram
\[
\xymatrix@C=10pt@R=15pt{
  H^0(W,\restr {T_{\mathbb{A}^{r+1}}} {W})
  \ar[rr] \ar@{=}[d]
  &&
  H^0(W,N_{W/\mathbb{A}^{r+1}})
  \ar@{=}[d]
  \\
  \bigoplus\limits_{k\in \mathbb{Z}} H^0(X,\cO_X(k+1))^{r+1}
  \ar[r]^(.54)\phi
  &
  \bigoplus\limits_{k\in \mathbb{Z}} H^0(X,  \restr {T_{\bP^r}}
  {X}(k))
  \ar[r]
  &
  \bigoplus\limits_{k\in \mathbb{Z}} H^0(X,N_{X/\bP^r}(k))
}
\]
(in which the map $\phi$
comes from the Euler exact sequence),
and \eqref{E:T12} is proved.
\end{proof}

Considering the degree $-1$ pieces of the exact sequences
\eqref{E:T12}, we get: 
    
\begin{cor}\label{C:exseq-1}
  In the notation of Proposition~\ref{P:schless},
  if \eqref{cond:aCM} or \eqref{cond:projnorm} holds,
then there is an exact sequence:
\begin{equation}\label{E:yy}
H^0(X, \restr {T_{\bP^r}} X (-1))\longrightarrow
H^0(X,N_{X/\bP^r}(-1))\longrightarrow T^1_{CX,-1}\to
0.
\end{equation}
\end{cor}

The following corollary will be important in our applications.
It applies in the cases under consideration in this article, because
embedded $K3$ surfaces, Gorenstein weighted projective spaces of
dimension $3$ in their anticanonical embedding, and their linear curve sections are arithmetically Gorenstein.

\begin{cor}\label{c:alpha-N-ineq}
Let $X \subset \P^r$ be either  aCM  of pure dimension $\ge 2$  or
arithmetically Gorenstein of pure dimension 1 and positive arithmetic genus. Then  $\alpha(X) \le \dim (T^1_{CX,-1})$.
\end{cor}

\noindent
(See Theorem~\ref{T:lvovski} for the definition of $\alpha$).

\begin{proof}
The twisted Euler exact sequence
\[
  0 \to \O_{X}(-1) \to
  H^0(X,\O(1)) \otimes \O_X \to
  \restr {T_{\P^r}} X (-1)
  \to 0
\]
induces the exact sequence
\begin{multline}\label{twEuler}
  0 \to H^0(X,\O(1))
  \to H^0 (\restr {T_{\P^r}} X (-1))
  \\\to
  H^1 (\O_{X}(-1))
  \to H^0(\O_X(1)) \otimes H^1(\O_X).
\end{multline}

When $\dim(X)>1$,
since $X$ is arithmetically Cohen--Macaulay we have
that $H^1 (\O_{X}(-1)) = 0$, and therefore
by \eqref{twEuler},
\[
  H^0 (\restr {T_{\P^r}} X (-1))
  \cong
  H^0(X,\O(1)).
\]
Then \eqref{E:yy} gives a presentation
\[
  H^0(X,\O(1))
  \longrightarrow
  H^0(X,N_{X/\bP^r}(-1))\longrightarrow T^1_{CX,-1}\to
  0,
\]
from which the desired inequality  follows at once. 

When $\dim(X)=1$,
\eqref{twEuler}
gives the following exact sequence of vector spaces,
\[
  0 \to H^0(X,\O(1))
  \to H^0 (\restr {T_{\P^r}} X (-1)) \to
  \ker (\trsp \mu) \to 0,
\]
where $\mu$ is the multiplication map
\[
  \mu: H^0(\O_X(1)) \otimes H^0(\omega_X) \to H^0(\omega_X(1)).
\]
If $X$ is arithmetically Gorenstein of positive genus we have
$\omega_X=\cO_X(\nu)$ for some $\nu \ge 0$, hence $\mu$ is the
multiplication map
$$
H^0(\O_X(1)) \otimes H^0(\cO_X(\nu)) \to H^0(\cO_X(\nu+1))
$$
which is surjective, and we conclude as before.
\end{proof}

If $X$ is smooth, in most cases the leftmost map of \eqref{E:yy} is
injective, so that in fact $\alpha(X) = \dim (T^1_{CX,-1})$.
The same holds in the cases under investigation in this article:

\begin{cor}\label{c:alpha-N-nice}
Let $X \subset \P^r$ be either a $K3$ surface with at most canonical
singularities, or a Gorenstein weighted projective $3$-space in its
canonical embedding. Then $\alpha(X) = \dim (T^1_{CX,-1})$.
\end{cor}

\begin{proof}
The kernel of the leftmost map of \eqref{E:yy} is contained in 
$H^0(X,T_X(-1))$.
If $X$ is a $K3$ surface then $H^0(X,T_X)=0$, hence
$H^0(X,T_X(-1))=0$ as well.
If $X$ is a Gorenstein weighted projective space, then
$H^0(X,T_X(-1))=H^3(X,\Omega^1_X)^\vee$ by Serre duality, hence it is
zero in this case as well. It follows that the leftmost map of
\eqref{E:yy} is injective.

On the other hand, it follows from the Euler exact sequence and the
vanishing of $H^1(X,\O_X(-1))$ that
$H^0(X, \restr {T_{\bP^r}} X (-1)) \cong H^0(X,\O_X(1))^\vee$
which has dimension $r+1$, hence the result.
\end{proof}

The following will also be of fundamental importance for us.

\begin{prop}\label{P:T1K3}
Let $X\subset \bP^g$ be a $K3$ surface with at worst canonical
singularities. Then
$T^1_{CX,-1}=\Ext^1_X(\Omega^1_X,\cO_X(-1))$.
\end{prop}

\begin{proof}
Taking
$\Hom(\ .\ , \O_X(-1))$
of the conormal exact sequence of $X$ in $\P^g$, and using the fact
that the conormal sheaf of $X$ in $\P^g$ and $\restr {\Omega^1_{\P^g}}
{X}$ are locally free, we obtain the
exact sequence
\begin{multline}
\label{E:yyprime}
H^0(X,\restr {T_{\bP^g}} {X}(-1))\longrightarrow
H^0(X,N_{X/\bP^g}(-1))\longrightarrow \Ext^1_X(\Omega^1_X,\cO_X(-1)) 
\\
\longrightarrow H^1(X,\restr {T_{\bP^g}} {X}(-1)).
\end{multline}
From the restricted and twisted Euler sequence
$$
0\longrightarrow \cO_X(-1)\longrightarrow
H^0(X,\cO_X(1))^\vee\otimes\cO_X\longrightarrow
\restr {T_{\bP^g}} {X} (-1)\longrightarrow 0
$$
we deduce that $H^1(X, \restr {T_{\bP^g}} {X} (-1))=0$.
Therefore comparing the two exact sequences
\eqref{E:yyprime} and \eqref{E:yy}
gives the assertion.
\end{proof}
    
     \begin{parag}\label{p:sweepout}
Consider now  an extension $\wt X$ of a projectively normal
$X\subset \bP^r$.
In such a situation, we let $e_{X/\tilde X} \in
\Ext^1(\Omega^1_X,\O_X(-1))$ be the class of the conormal exact sequence
\[
\xymatrix{
0\ar[r]& \O_X(-1) \ar[r] 
& \bigl. {\Omega^1_{\tilde X}} \bigr|_X  \ar[r]& \Omega^1_X \ar[r]&0.
}\]

If the extension $\tilde X$ is non-trivial, \ie it is not a cone
over $X$, then we can also associate to it a family of deformations of $\ol{CX}$, the
projective cone over $X$, as follows. Let $X=\wt X \cap H$, where
$H\cong \bP^r\subset \bP^{r+1}$ is a hyperplane. Consider in
$\bP^{r+2}$ the projective cone $\ol{C\wt X}$ and the pencil of
hyperplanes $H_t$ with center $H$.  Let $H_o$ be the hyperplane
containing the vertex $v$ of $\ol{C\wt X}$. Then $H_o\cap \ol{C\wt
X}=\ol{CX}$, while $H_t\cap \ol{C\wt X}\cong \wt X$ for all $t\ne
o$. After blowing up $X$ we obtain a family
   $$
   f:\mathrm{Bl}_X(\ol{C\wt X}) \longrightarrow \bP^1
   $$
   which is flat because $\wt X$ is projectively normal, with
$f^{-1}(t)=H_t\cap \ol{C\wt X}$. By restriction we get a deformation
of the affine cone $CX$.
If $\wt X$ is smooth then this deformation is a smoothing of
$\ol{CX}=f^{-1}(o)$.  This is a classical construction called
\emph{sweeping out the cone} (see, e.g., \cite[(7.6)(iii)]{hP74}).
Algebraically the above construction has the following
description. Let $\wt X= \mathrm{Proj}(\c A)$, where $\c
A=\mathbb{C}[X_0,\dots,X_r,t]/J$. Then
 $$
 A= \c A/t\c A = \mathbb{C}[X_0,\dots,X_r]/I
 $$
 where $I=J/tJ$. Consider $C\wt X=\mathrm{Spec}(\c A)\subset
\mathbb{A}^{r+2}$. The pencil of parallel hyperplanes $V(t)\subset
\mathbb{A}^{r+2}$ has as projective closure the pencil $\{H_t\}$
considered before. Therefore the morphism
 $$
 \phi:\mathrm{Spec}(\c A) \longrightarrow \mathrm{Spec}(\mathbb{C}[t])
 $$
is the corresponding family of deformations of $CX$.
It is clear that
if $e_{X/\wt X}\in T^1_{CX,-1}$ (e.g., $X$ is nonsingular or is a
singular $K3$ surface) then the first order deformation of $X$
associated to $\phi$ is
$e_{X/\wt X}$.
Note that, by construction, $e_{X/\wt X}$ \emph{is
unobstructed both as a first order deformation of $CX$ and of
$\ol{CX}$.}
    \end{parag}

The upshot of the above construction is that the datum of an extension
$\tilde X$ of $X$ gives a deformation of the cone over $X$. In fact
the two objects correspond to the same ring $\A$: the former is
$\mathrm{Proj}(\A)$ and the latter is $\mathrm{Spec}(\A)$.
We shall now state a result of Wahl which will be crucial in what
follows.  It may be
considered as a reverse sweeping out the cone, in that it produces an
extension of $X$ from a first order deformation of the cone over $X$.
We first need the following standard definition.

\begin{definition}\label{def:N2}
Let $X\subset \bP^r$  be a non-degenerate projectively normal   scheme
of pure dimension $\ge 1$.
we say that $X$ \emph{has the property $N_2$} or \emph{satisfies
  $N_2$}
if its homogeneous coordinate
ring $A=R/I_X$  has a minimal graded presentation over
$R:=\mathbb{C}[X_0,\dots,X_r]$ of the form: 
\begin{equation}\label{E:presentA1}
\xymatrix{
R(-3)^a \ar[r]^-\psi&R(-2)^b \ar[r]^-\phi & R\ar[r]&A \to 0.
}    
\end{equation}
\end{definition}

 \begin{thm}[\cite{jW97},   Proof of Thm.\ 7.1 and  Remark 7.2]\label{T:wahl}
   Let $X\subset \bP^r$  be a non-degenerate projectively normal
   scheme of pure dimension $\ge 1$ and let $A$ be its homogeneous
   coordinate ring. Consider the following two conditions:
   \begin{itemize}
       \item[(a)] $X$ has the property $N_2$;
       \item[(b)]  $T^2_{A,k}=0$ for all $k \le -2$.
   \end{itemize}
   
If (a) holds then 
any first order deformation of $CX$ of degree
$-1$ lifts to at most one  graded deformation $\mathcal{A}$
over $\mathbb{C}[t]$, with $\deg(t)=1$.  Moreover
$Y:=\mathrm{Proj}(\mathcal{A})\subset
\bP^{r+1}=\mathrm{Proj}[t,X_0,\dots, X_r]$ is an extension of
$X:=\mathrm{Proj}(A)=\mathrm{Proj}(\mathcal{A})\cap
\{t=0\}\subset \bP^r$ which is unique up to projective
automorphisms of $\bP^{r+1}$ fixing every  point of $\bP^r= \{t=0\}$.   
       
If both (a) and (b) hold then every
first order deformation of $CX$ of degree
$-1$ lifts to a graded deformation $\mathcal{A}$ as above.
   \end{thm}

It is one of the main results of \cite{ABS17} that condition (b) above
holds when $X$ is a canonical curve.
We shall use this and Theorem~\ref{T:CDS} to prove that the same holds
when $X$ is a Gorenstein \wps{} of dimension 3,
see Corollary~\ref{c:extend}.

For more details on the unicity statement, we refer to
\cite[Rmk.~4.8]{CDS20}. Note that assumption (a) implies that 
$H^0(N_{X/\P^r} (-k))=0$ for all $k\geq 2$, as we have already
observed.

   

\section{The deformation argument}\label{S:thearg}

We now come to our main technical result, and its application to
deformations of weighted projective spaces.

\begin{parag}
Let $(S,L)$ be a polarized
$K3$ surface with canonical singularities,
and $g=h^0(L)-1$.
A \emph{smoothing} of $(S,L)$ is a pair
$\bigl(p:\cS\to (\Delta,o), \L \bigr)$,
where $p$ is a smoothing of $S$ over an affine nonsingular pointed
curve $(\Delta,o)$ and $\L$ extends $L$, \ie
$L=\L(o):=\restr {\L}{p^{-1}(o)}$. 
There is a flat family of surfaces in $\bP^g$ associated to such
a smoothing:
  \begin{equation}\label{E:familycS}
 \xymatrix{\cS \ar@{^(->}[r]^-j\ar[dr]_-p& \bP^g\times \Delta\ar[d]^-{pr_2}\\
 &\Delta 
 }
  \end{equation}
where $j$ is defined by the sections of $\L$.  

We shall use the following notation:
  $\Delta^\circ:= \Delta\setminus \{o\}$; \quad $\cS^\circ
  =\cS\setminus p^{-1}(o)$; \quad
  $p^\circ=\restr p {\cS^\circ}$;   \quad
  $\L^\circ = \restr {\L} {\cS^\circ}$. 

A \emph{relative extension} of $\cS\subset \bP^{g}\times \Delta$
consists of an $\X\subset \bP^{g+1}\times\Delta$, flat over
$\Delta$, together with a relative hyperplane $\c H\cong
\bP^{g}\times\Delta\subset \bP^{g+1}\times\Delta$ such that
$\X \cap\c H =\cS$ and $\X(t)$ is not a cone over $\cS(t)$ for all
$t\in \Delta$.
  Similarly, one defines relative extensions   of $\cS^\circ\subset \bP^{g}\times \Delta^\circ$.
\end{parag}

\begin{thm}\label{T:smoothing}
Let $S_0 \subset \P^g$ be a $K3$ surface, possibly with canonical
singularities, and $V_0 \subset \P^{g+1}$ be an 
extension of
$S_0$. 
Let $p: \cS \to \Delta$ be a smoothing of $S_0$ in
$\P^g$ as above, and assume that the following conditions hold:
\begin{compactenum}[(a)]
\item\label{ass:CDS} $g\geq 11$, and for all $t\in \Delta^\circ$ we
  have $\Cliff(S_t)>2$; 
\item\label{ass:N2} $S_0$ has the $N_2$ property;
\item\label{ass:ct} $t\in \Delta \mapsto \alpha(S_t)$ is constant.
\end{compactenum}
Then there exists a deformation of $V_0$ in $\P^{g+1}$ which is a
relative extension of $\cS \subset \P^g\times \Delta$.
\end{thm}

\begin{proof}
We have a base change map \cite{hL83}: 
$$
\tau(o): \sExt^1_p(\Omega^1_{\cS/\Delta},\L^{-1})_o\otimes k(o)
\longrightarrow \Ext^1_{S_0} \bigl( \Omega^1_{S_0}, L_0^{-1} \bigr)
$$
with $L_0=\L(o) = \O_{S_0}(1)$
(note that $\Omega^1_{\cS/\Delta}$ is  $\Delta$-flat because
$p$ is flat and has reduced fibres).
By our assumption \eqref{ass:ct} and the results in \S
\ref{S:extendability}, the function 
 $$
t \in \Delta \longmapsto  \dim \bigl[
 \Ext^1_{\cS(t)}(\Omega^1_{\cS(t)},\L(t)^{-1}) \bigr]
$$
is constant, hence
$\sExt^1_p(\Omega^1_{\cS/\Delta},\L^{-1})$ is locally free and
$\tau(o)$ is an isomorphism. It follows that there exists
a section $E\in
\sExt^1_p(\Omega^1_{\cS/\Delta},\L^{-1})$ such that
$\tau(o)(E)=e_{S_0/V_0}$
(see \ref{p:sweepout} for the definition of $e_{S_0/V_0}$).

For all $t \in \Delta^\circ$, the smooth $K3$ surface $\cS(t)\subset
\bP^{g}$ satisfies the assumptions of Theorem~\ref{T:CDS} by
\eqref{ass:CDS}, and therefore there exists  a unique extension
$\cS(t)\subset\X(t)\subset\bP^{g+1}$ such that
$e_{\cS(t)/\X(t)}=E(t)$.
We then consider
$\X^\circ :=\bigcup_{t\in\Delta^\circ}\X(t)$:
it is a relative extension of $\cS^\circ$,
and its Zariski closure
$\X=\ol{\X^\circ}\subset \bP^{g+1}\times \Delta$
is a relative extension of $\cS$.

Let $X_0=\X(o)$. One has
$e_{S_0/X_0} = E(o) = e_{S_0/V_0}$,
so assumption \eqref{ass:N2} and Theorem~\ref{T:wahl} imply that
$X_0=V_0$, which ends the proof.
\end{proof}

We now set up the situation in which we will apply the above
Theorem~\ref{T:smoothing}. The notation is the same as in
\ref{p:Gwps}.

\begin{parag}\label{p:application}
Consider $\bP=\bP(a_0,a_1,a_2,a_3)$ a \wps{} with Gorenstein canonical
singularities, and $(S,L)$ a general anticanonical divisor of $\P$,
so $L = - \restr {K_\P} S$. Let $i_S$ be the divisibility of $L$ in
$\Pic(S)$, and
$L_1$ be the primitive line bundle on $S$ such that
$L=i_S L_1$.
Thus $(S,L_1)\in \Kprim_{g_1}$, where $g_1=h^0(L_1)-1$. 

We may then consider a deformation 
$\bigl(p:\cS\to (\Delta,o), \L_1 \bigr)$
of $(S,L_1)$ to general primitive polarized smooth $K3$
surfaces of genus $g_1$. 
To such a smoothing there is associated a flat family of surfaces in
$\bP^{g_1}$ as in \eqref{E:familycS}, and also
an analogous family in $\P^g$ defined by the sections of $\L_1 ^{i_S}$:
\begin{equation}\label{E:familycS2}
 \xymatrix{\cS \ar@{^(->}[r]^-{j_{i_S}}\ar[dr]_-p& \bP^{g_\bP}\times \Delta\ar[d]^-{pr_2}\\
 &\Delta 
 }
\end{equation}
where $g=h^0(S,L)-1$
($j_{i_S}$ is the $i_S$-uple Veronese re-embedding of $\cS$).

We shall apply Theorem~\ref{T:smoothing} to $S_0=S \subset \P^g$, and
$V_0 = \P \subset \P^{g+1}$ in its anticanonical embedding.
In this case assumption \eqref{ass:CDS} is always satisfied, as a
direct computation shows.
Assumption \eqref{ass:N2} is always satisfied as well, because $\P$
has the property $N_2$ by Proposition~\ref{p:N2} below.
Assumption \eqref{ass:ct} however does not hold in all cases: we
compute $\alpha(S,L)$ in Proposition~\ref{p:alpha-val} below, and
compare it with $\alpha(S',L')$ for a general $(S',L') \in \K_g^{i_S}$,
in other words with $\alpha(S',(L_1')^{i_S})$ for a general
$(S',L_1') \in \Kprim_{g_1}$.

When $\alpha(S,L)=\alpha(S',L')$ holds, we conclude that $\P$ deforms
to a threefold extension of a general $K3$ surface $(S',L') \in
\K_g^{i_S}$. 
This happens exactly for cases $\#i$,
$i \in \{ 1, \dots, 7, 9\}$, see Table~\ref{tab:alpha} below. 
We refer to Section~\ref{S:examples} for a more precise
description of the output in each of these cases.
\end{parag}


\section{Explicit computations on WPS}\label{S:comput}

The main object of this section is to analyze which Gorenstein
projective spaces enjoy the required properties for
Theorem~\ref{T:smoothing} to apply, as described in
\ref{p:application} above.
We carry this out by explicit computations using the software
Macaulay2 \cite{M2}.
As a bonus, we obtain the number of times each Gorenstein weighted
projective $3$-space is extendable.

\begin{prop}
\label{p:N2}
Let $\P \subset \P^{g+1}$ be a $3$-dimensional Gorenstein
weighted projective space, considered in its anticanonical embedding.
Then $\P$ is projectively normal, and its homogeneous coordinate ring $A=R/I_\P$  has a
minimal resolution of the form  

\[
  \cdots \longrightarrow
  R(-3)^{\beta_2} \longrightarrow
  R(-2)^{\beta_1} \longrightarrow
  R  \longrightarrow A \longrightarrow 0
\]
with $\beta_1,\beta_2$ as indicated in Table~\ref{tab:N2} below. 
In particular $\P$ has the $N_2$ property.
\end{prop}

\begin{center}
\begin{tabular}{|l|l|l|l|l|l|l|}
\hline
\# & weights        & $g_1$ & \raisebox{0mm}[3.25mm]{$i_S$} &  $g$ & $\beta_1$ & $\beta_2$ \\
\hline
 1 & (1, 1, 1, 3)   &   2 & 6 & 37 &     595 &   13056 \\
 2 & (1, 1, 4, 6)   &   2 & 6 & 37 &     595 &   13056 \\
 3 & (1, 2, 2, 5)   &   2 & 5 & 26 &     276 &    4025 \\
 4 & (1, 1, 1, 1)   &   3 & 4 & 33 &     465 &    8960 \\
 5 & (1, 1, 2, 4)   &   3 & 4 & 33 &     465 &    8960 \\
 6 & (1, 3, 4, 4)   &   3 & 3 & 19 &     136 &    1344 \\
 7 & (1, 1, 2, 2)   &   4 & 3 & 28 &     325 &    5175 \\
 8 & (1, 2, 6, 9)   &   4 & 3 & 28 &     325 &    5175 \\
 9 & (2, 3, 3, 4)   &   4 & 2 & 13 &      55 &     320 \\
10 & (1, 4, 5, 10)  &   6 & 2 & 21 &     171 &    1920 \\
11 & (1, 2, 3, 6)   &   7 & 2 & 25 &     253 &    3520 \\
12 & (1, 3, 8, 12)  &   7 & 2 & 25 &     253 &    3520 \\
13 & (2, 3, 10, 15) &  16 & 1 & 16 &      91 &     715 \\
14 & (1, 6, 14, 21) &  22 & 1 & 22 &     190 &    2261 \\
\hline
\end{tabular}
\captionof{table}{First Betti numbers of Gorenstein weighted
  projective spaces}
\label{tab:N2}
\end{center}

Of course $\beta_1 = \binom {g-2} 2$, since curve linear sections of
$\P$ are canonical curves of genus $g$.

\begin{proof}
The projective normality follows from the fact that $\P$ has canonical
curves as linear sections, see \cite[Thm.~5.1]{CDS20}.
For property $N_2$,
we explicitly compute the ideal of $\P$ in $\P^{g+1}$ using
Macaulay2, then compute the first syzygies of this ideal, and
eventually check that they are of the asserted shape.
This computation goes as follows.

Let $\P=\P(a_0,a_1,a_2,a_3)$ endowed with weighted homogeneous coordinates
$\bx = (x_0,x_1,x_2,x_3)$.
First, one writes down the list $(M_0,\ldots,M_{g+1})$ of all 
monomials in $\bx$ of weighted degree $s=a_0+a_1+a_2+a_3$, which form
a basis of $H^0(\P,-K_\P)$.
Then the ideal of the graph $\Gamma \subset \P \times \P^{g+1}$ of
the embedding $\P \subset \P^{g+1}$ is
\[
  I_\Gamma = \bigl(
  y_i-M_i(\bx),
  \quad i=0,\dots,g+1 \bigr),
\]
with $(y_0,\dots,y_{g+1})$ homogeneous coordinates on $\P^{g+1}$.
One obtains the ideal $I_\P$ of $\P \subset \P^{g+1}$ by eliminating
$\bx$ from $I_\Gamma$, which may be performed efficiently using a
Gröbner basis algorithm.
Eventually, there is a Macaulay2 function which computes step by step
the syzygies of this ideal.
We provide the explicit Macaulay2 commands implementing this procedure
at the end of this article.
\end{proof}

In principle one may use any basis of $H^0(\P,-K_\P)$ to compute the
ideal, but the computations turn out to work faster with a monomial
basis. In fact doing so one takes advantage of $\P$ being a toric
variety. There is also a Macaulay2 function computing the whole
resolution of a graded ideal, but we have not been able to run these
computations successfully for $I_\P$ (apart for \#13) because the
complexity was too large.

In principle it is possible to compute all Betti numbers of any
lattice ideal $I_\Lambda$ as the dimensions of the reduced homology
groups of a simplicial complex explicitly construct from the lattice
$\Lambda$, see, e.g., \cite[Thm.~9.2]{MS}
or \cite[Chap.~5]{MT+}. It seems to us however that this leaves
non-trivial computations to be performed, which we haven't tried to
carry out.

\begin{prop}\label{p:alpha-val}
Let $\P \subset \P^{g+1}$ be a $3$-dimensional Gorenstein
weighted projective space in its anticanonical embedding, and $(S,L)$
be a general hyperplane section of $\P$.
We write $i_S$ for the divisibility of $L=\restr {-K_\P} S$ in $\Pic(S)$.
Let $(S',L')$ be a general member of $\K_g^{\iota_S}$.
Then the values of $\alpha(S)$ and $\alpha(S')$ are as indicated in
Table~\ref{tab:alpha} below.
Moreover, $\alpha(\P)=\alpha(S)-1$, and $\alpha(C)=\alpha(S)+1$ for
$C$ a general curve linear section of $\P$.
\end{prop}

\begin{center}
\begin{tabular}{|l|l|l|l|l||l|l|}
\hline
\raisebox{0mm}[3.5mm]{\#}
& weights        & $g_1$ & \raisebox{0mm}[3.25mm]{$i_S$} & $\alpha(S)$ & $\alpha(S')$ & 3-fold\\
\hline
  1 & (1, 1, 1, 3)   &   2 & 6 &         1  &       1                                    & \raisebox{0mm}[4mm]{$\P(1^3,3)$}\\
  2 & (1, 1, 4, 6)   &   2 & 6 &         1  &       1 & $\P(1^3,3)$\\
  3 & (1, 2, 2, 5)   &   2 & 5 &         3  &       3                                                  & $H_6 \subset \P(1^3,3,5)$\\
  4 & (1, 1, 1, 1)   &   3 & 4 &         1  &       1 & $\P^3$\\
  5 & (1, 1, 2, 4)   &   3 & 4 &         1  &       1 & $\P^3$\\
  6 & (1, 3, 4, 4)   &   3 & 3 &         4  &       4                                                     & $H_4 \subset \P(1^4,3)$\\
  7 & (1, 1, 2, 2)   &   4 & 3 &         1  &       1 & $\Q$\\
  8 & (1, 2, 6, 9)   &   4 & 3 &         2  &       1 & $\Q$\\
  9 & (2, 3, 3, 4)   &   4 & 2 &         6  &       6                                                   & $H_3 \subset \P^4$\\
 10 & (1, 4, 5, 10)  &   6 & 2 &         3  &       1 & $V_5$\\
 11 & (1, 2, 3, 6)   &   7 & 2 &         1  &       0 & does not exist\\
 12 & (1, 3, 8, 12)  &   7 & 2 &         2  &       0 & does not exist\\
 13 & (2, 3, 10, 15) &  16 & 1 &         3  &       0 & does not exist\\
 14 & (1, 6, 14, 21) &  22 & 1 &         2  &       0 & does not exist\\
\hline
\end{tabular}
\captionof{table}{Dimension of the weight $-1$ piece of $T^1$}
\label{tab:alpha}
\end{center}

In the table, we also indicate the general $3$-fold extension of $S'$,
with the following notation:
$\Q$ denotes the smooth $3$-dimensional quadric in $\P^4$; $H_d$
denotes a general degree $d$ hypersurface in the specified projective
space; 
$V_5$ denotes the degree $5$ Del Pezzo threefold, \ie the section of
the Grassmannian $\G(2,5)$ by a general $\P^6$ in the Plücker
embedding.
We refer to \cite{CDdouble,CDhigher} for these matters.

We will need the following lemma for the proof, which is a
generalization of a well-known fact when all involved varieties are
smooth. 

\begin{lemma}\label{l:alpha-hyp}
Let $\P\subset \P^{g+1}$ be a Gorenstein weighted projective space,
$S$ a hyperplane section of $\P$, and $C$ a hyperplane
section of $S$. Then one has
\[
  \alpha(C) \geq \alpha(S)+1 \geq \alpha(\P)+2.
\]
\end{lemma}

\begin{proof}
We first compare $\alpha(\P)$ and $\alpha(S)$.
Since $S$ is a hyperplane section of $\P$, one has
$\restr {N_{\P/\bP^{g+1}}} S = N_{S/\bP^g}$.
We thus have the following exact sequence, where the rightmost map is
the restriction map:
\begin{equation}\label{restr-normal}
  0\to      N_{\P/\bP^{g+1}}(-2) \longrightarrow N_{\P/\bP^{g+1}}(-1)
  \longrightarrow N_{S/\bP^g}(-1) \to 0
\end{equation}
with $\O(1)$ the line bundle induced by the embedding in $\P^{g+1}$.
By Proposition~\ref{p:N2}, $\P \subset \P^{g+1}$ has the property
$N_2$, hence $H^0(N_{\P/\bP^{g+1}}(-2))=0$
(see \cite[Lem.~1.1]{aK20} and the references therein).
So the long exact sequence induced by \eqref{restr-normal} shows the
inequality
\begin{equation*}
  h^0 \bigl( N_{\P/\bP^{g+1}}(-1) \bigr)
  \leq
  h^0 \bigl( N_{S/\bP^g}(-1) \bigr).
\end{equation*}
By the definition of $\alpha$ in Theorem~\ref{T:lvovski}, this ends
the proof.
The inequality between $\alpha(S)$ and $\alpha(C)$ is obtained in the
same way.
\end{proof}

\begin{proof}[Proof of Proposition~\ref{p:alpha-val}]
We know the ideal $I_\P$ of $\P\subset \P^{g+1}$ from the proof of
Proposition~\ref{p:N2}. Using the Macaulay2 package
``VersalDeformations'' \cite{M2T1} one can then compute
$\dim (T^1_{C\P,-1})$, and this equals $\alpha(\P)$ by 
Corollary~\ref{c:alpha-N-nice}.

Next we choose two explicit (see below)
linear functionals $l_0$ and $l_1$ defining
hyperplanes $H_0$ and $H_1$ in $\P^{g+1}$, and consider
$S_0 = \bP \cap H_0 \subset \P^{g}$ and $C_0 = S_0 \cap H_1 \subset
\P^{g-1}$. Using the same procedure we compute
\(
   \dim (T^1_{CS_0,-1})
\)
and
\(
   \dim (T^1_{CC_0,-1}),
\)
and find out that
\[
  \dim (T^1_{CS_0,-1}) = \alpha(\P)+1
  \quad \text{and}
  \quad
  \dim (T^1_{CC_0,-1}) = \alpha(\P)+2.
\]
Again, the explicit Macaulay2 commands implementing this procedure
are given at the end of this article.

Let $S$ be a general hyperplane section of $\P$. Then on the one hand
one has $\alpha(S) \geq \alpha(\P)+1$
by Lemma~\ref{l:alpha-hyp} above,
and on the other hand one has
$\alpha(S) \leq \alpha(S_0)$
by semicontinuity since
$\alpha(S)=h^0(N_{S/\P^g}(-1))-g-1$ by definition, and
$\alpha(S_0)\le \dim (T^1_{CS_0,-1})$ by Corollary~\ref{c:alpha-N-ineq}.
Hence $\alpha(S)=\alpha(\P)+1$.
Similar reasoning yields $\alpha(C)=\alpha(\P)+2$ for a general curve
linear section of $\P$.
\end{proof}

In practice, if one chooses random linear functionals $l_0$ and $l_1$
then the complexity of the computation of the weight $-1$ piece of
$T^1$ is too high and one cannot get an answer. We chose
\[
  l_0=x_7+x_{g+1}
  \quad \text{and}
  \quad
  l_1=x_3+x_{g},
\]
so that the corresponding linear sections are again toric; in
particular $S_0$ is not a $K3$ surface and $C_0$ is singular.
This is the reason why we have to resort to
Corollary~\ref{c:alpha-N-ineq} in the proof; note that we cannot
guarantee either that
$\dim (T^1_{CS,-1})$ is semi-continuous as $S$ approaches $S_0$.
In principle Macaulay2 can compute
$h^0(N_{S_0/\P^g}(-1))$ directly, but in practice it is not able to
return an answer.

\begin{cor}
\label{c:extend}
Let $\P \subset \P^{g+1}$ be a $3$-dimensional Gorenstein
weighted projective space in its anticanonical embedding.
Then $\P$ is extendable exactly $\alpha(\P)$ times.
\end{cor}

\noindent
(Recall that $\alpha(\P)=\alpha(S)-1$ with $\alpha(S)$ as in
Table~\ref{tab:alpha}).

\begin{proof}
First note that by Lvovski's Theorem~\ref{T:lvovski},
applied to a general (smooth) curve linear section of $\P$,
$\P \subset \P^{g+1}$ is at most $\alpha(\P)$-extendable.
To prove the converse, let us consider $C$ a general curve linear
section of $\P$. It is a smooth canonical curve of genus $g\geq 11$
and Clifford index strictly larger than $2$, hence liable to
Theorem~\ref{T:CDS}. So there exists a universal extension of $C$,
which is an $(\alpha(C)+1)$-dimensional variety $X \subset
\P^{g-1+\alpha(C)}$,
\ie an $(\alpha(\P)+3)$-dimensional variety $X \subset
\P^{g+1+\alpha(\P)}$.

The pencil of hyperplanes in $\P^{g+1}$ containing $C$ cuts out on
$\P$ a pencil of $K3$ surfaces, which are not all isomorphic by
\cite[Prop.~1.7]{Pa}
(as observed in \cite{CDGK},
the latter statement in fact applies to all varieties
different from cones).
By the universality of $X$,
this implies that $\P$ is a linear section of $X$, hence it is
$\alpha(\P)$-extendable.
\end{proof}


\section{Examples}\label{S:examples}

In this section we describe explicitly the output of
Theorem~\ref{T:smoothing} 
and make additional remarks.
We first list the cases to which Theorem~\ref{T:smoothing}
applies; see also Remarks~\ref{r:prokh-degen} and
\ref{r:weighted-trick} for another point of view on these examples.
The notation is that of Table~\ref{tab:alpha}.

\begin{example}[\#1 and \#2]
The general member of $\K_{37}^6$ extends to $\P(1^3,3)$, hence the
application of Theorem~\ref{T:smoothing} to \#1 is trivial.
On the other hand the application to \#2 tells us that there exists a
deformation of $\P(1,1,4,6)$ to $\P(1^3,3)$.
Note that these are the only Fano varieties with canonical Gorenstein
singularities of genus $37$, the maximal possible value, by
\cite{yP05}. $\P(1^3,3) \subset \P^{38}$ is the $2$-Veronese
reembedding of the cone in $\P^{10}$ over the Veronese variety
$v_2(\P^3)$; in particular it is rigid. Thus the deformation of 
$\P(1,1,4,6)$ to $\P(1^3,3)$  exhibits a jump phenomenon.
\end{example}

\begin{example}[\#3]\label{ex2}
Theorem~\ref{T:smoothing} tells us in this case that 
$\P(1,2,2,5) \subset \P^{27}$ deforms to a general $6$-ic hypersurface
$H_6 \subset \P(1^3,3,5)$ in its anticanonical embedding by $\O(5)$.
Such an $H_6$ is singular, and its singularities may be listed
following \cite{fletcher}; in particular
as $5\! \not\kern -.1mm|\ 6$,
$H_6$ passes through the point $P_4=(0:0:0:0:1)$ and one finds it has
a quotient singularity of type $\frac 15 (1,1,3)$ there.

Corollary~\ref{c:extend} tells us that $\P(1,2,2,5) \subset \P^{27}$
is $2$-extendable, as is $H_6 \subset \P(1^3,3,5)$.
The same deformation argument as that given to prove
Theorem~\ref{T:smoothing} shows that the $2$-extension of
$\P(1,2,2,5)$ deforms to that of $H_6$, which is a sextic hypersurface
$\tilde H_6 \subset \P(1^3,3,5^3)$ embedded by $\O(5)$, that is
$- \frac 13 K_{\tilde H_6}$, see \cite{CDdouble}.
\end{example}

\begin{example}[\#4 and \#5]\label{ex3}
\#4 is of course the Veronese variety $v_4(\P^3)$, which is rigid and
extends the general member of $\K_{33}^4$; Theorem~\ref{T:smoothing}
is trivial in this case.
The application to \#5 however tells us that
$\P(1,1,2,4) \subset \P^{34}$ smoothes to $v_4(\P^3)$,
which may be seen elementarily as follows.

Spelling out a monomial basis of $H^0(\P(1,1,2,4), \O(4))$, one sees
that $\O(4)$ induces an embedding of $\P(1,1,2,4)$ as a cone over
$\P(1,1,2)$ embedded by its own $\O(4)$, with
vertex a point, in $\P^9$.
In turn $\P(1,1,2,4) \subset \P^{34}$ is embedded by $\O(8)$, hence it
is the $2$-Veronese reembedding of the latter cone in $\P^9$.
In the same way, the embedding
of $\P(1,1,2)$ by $\O(4)$ is the $2$-Veronese reembedding of a quadric
cone (of rank $3$) in $\P^3$.

Thus in the embedding by $\O(4)$, $\P(1,1,2,4)$ is the cone over a
section of the Veronese variety $v_2(\P^3)$ by a tangent
hyperplane. This deforms to the cone over a section by a transverse
hyperplane (this corresponds to smoothing the quadric in $\P^3$ image
of $\P(1,1,2)$ by $\O(2)$). In turn, this deforms to the Veronese
variety $v_2(\P^3)$ itself by ``sweeping out the cone'' (see
\ref{p:sweepout}). In its anticanonical embedding, $\P(1,1,2,4)$
correspondingly deforms to the $2$-Veronese re-embedding of
$v_2(\P^3)$, which is the Veronese variety $v_4(\P^3)$.
\end{example}

\begin{example}[\#6]\label{ex4}
This case is similar to \#3 and we will be brief.
Theorem~\ref{T:smoothing} provides a deformation of
$\P(1,3,4,4) \subset \P^{20}$ to the anticanonical embedding by
$\O(3)$ of a general $4$-ic $H_4 \subset \P(1^4,3)$. The latter is
singular; in particular as $3\! \not\kern -.1mm|\ 4$,
$H_4$ always passes through the coordinate point $P_4$ and has a
quotient singularity of type $\frac 13(1,1,1)$ there,
\ie, it is locally isomorphic to the cone over
the Veronese variety $v_3(\P^2)$.

The argument of Theorem~\ref{T:smoothing} shows that the $3$-extension
of $\P(1,3,4,4) \subset \P^{20}$ deforms to that of $H_4$, which is a
$4$-ic hypersurface $\tilde H_4 \subset \P(1^4,3^4)$ embedded by
$\O(3)$, see \cite[\S 3]{CDhigher}.
\end{example}

\begin{example}[\#7]\label{ex5}
Theorem~\ref{T:smoothing} provides a smoothing of 
$\P(1,1,2,2) \subset \P^{29}$ to a smooth quadric $\Q$, in its
canonical embedding. This smoothing may be elementarily found, noting
(as we did for case \#5) that $\O(2)$ realizes $\P(1,1,2,2)$ as a rank
$3$ quadric in $\P^4$. 
\end{example}

\begin{example}[\#9]\label{ex6}
This case is similar to \#3 and \#6, and in fact easier, so we will be
very brief.
Theorem~\ref{T:smoothing} proves that $\P(2,3,3,4)$ deforms to a
general cubic hypersurface $H_3$ in $\P^4$, in particular this is a
smoothing. The $5$-extension of $\P(2,3,3,4)$ deforms to that of
$H_3$, which is a complete intersection $\tilde H_2 \cap \tilde H_3$
in $\P(1^5,2^6)$, see \cite[\S 3]{CDhigher}.
\end{example}

\begin{rem}
\label{r:prokh-degen}
The degeneration of $\P(1^3,3)$ to $\P(1^2,4,6)$ may be seen
explicitly as follows; this has been shown to us by the referee,
inspired by \cite[\S 11]{Hacking}.
The weighted projective space $\P(1^2,4,6)$ is $\Proj(R)$ with $R$ the
graded algebra $\C[x,y,z,w]$ in which $x,y,z,w$ have
respective weights $1,1,4,6$.
It is isomorphic to $\Proj(R^{(2)})$ where $R^{(2)}$ is the algebra
determined by $\O(2)$ on $\P(1^2,4,6)$, i.e.,
the graded piece $R^{(2)}_n$ is $R_{2n}$ for all $n\in \Z$ by definition.

We claim that $\Proj(R^{(2)})$ is naturally a
quadric in $P(1^3,2,3)$.
To see this we note that
$R^{(2)}$ is generated as a $\C$-algebra by
\[
  x^2,xy,y^2,z,w,
\]
which have weights $2,2,2,4,6$ in $R$, hence $1,1,1,2,3$ in
$R^{(2)}$.
The only relation between them is $x^2\cdot y^2 = (xy)^2$, so
\[
  R^{(2)} \cong
  \frac {\C[a,b,c,u,v]}
  {(ac-b^2)},
\]
the isomorphism being given by mapping
$x^2,xy,y^2,z,w$ to $a,b,c,u,v$ respectively.
Therefore $\Proj(R^{(2)})$ is the quadric $ac=b^2$ in
$\Proj(\C[a,b,c,u,v])=\P(1^3,2,3)$.

The degeneration is then gotten by noting that $\P(1^3,3)$ is the
quadric $u=0$ in $\P(1^3,2,3)$. In the pencil of quadrics
\[
  ac-b^2+\lambda u=0,
\]
the member given by $\lambda=0$ is $\P(1^2,4,6)$ and all the others
are isomorphic to $\P(1^3,3)$.
\end{rem}

\begin{rem}
\label{r:weighted-trick}
In fact, following a suggestion of the referee, we have found that all
our examples above may be understood as in the previous
Remark~\ref{r:prokh-degen}. 
In general we shall consider the $d$-Veronese embedding, i.e., the
graded ring $R^{(d)}$, with $d=s/i_S$, where $s$ is the
sum of the weights so that $\omega_\P^{-1}=\O(s)$, and $i_S$ is as in
Table~\ref{tab:alpha}.
Let us briefly indicate the explicit computations.
We take $R=\C[x,y,z,w]$ the graded algebra giving the weighted
projective space under consideration, as in
Remark~\ref{r:prokh-degen}. 

\emph{Example~\ref{ex2}:}
$\P(1,2,2,5)$ is itself a sextic hypersurface in $P(1^3,3,5)$.
Indeed the algebra $R^{(2)}$ is generated by
\[
  x^2,y,z,xw,w^2
\]
which have weights $2,2,2,6,10$ in $R$, hence $1,1,1,3,5$ in
$R^{(2)}$.
The only relation between these generators is
$x^2\cdot w^2=(xw)^2$, which is in
weight $6$ in $R^{(2)}$, so that $\Proj(R^{(2)})$ is naturally a sextic
hypersurface in $P(1^3,3,5)$. 

\emph{Example~\ref{ex3}:} $\P(1^2,2,4)$ is isomorphic to a quadric in
$\P(1^4,2)$, hence it is a degeneration of $\P^3$. Indeed 
the algebra $R^{(2)}$ is generated by
\[
  x^2,xy,y^2,z,w
\]
which have weights $2,2,2,2,4$ in $R$, hence $1,1,1,1,2$ in
$R^{(2)}$.
The only relation between these generators is
$x^2\cdot y^2=(xy)^2$, which has
weight $2$ in $R^{(2)}$.

\emph{Example~\ref{ex4}:} $\P(1,3,4,4)$ is itself a quartic in $\P(1^4,3)$.
Indeed the algebra $R^{(4)}$ is generated by
\[
  x^4,xy,y^4,v,w
\]
which have weights $4,4,12,4,4$ in $R$, hence $1,1,3,1,1$ in
$R^{(4)}$.
The only relation between these generators is
$x^4\cdot y^4=(xy)^4$, which has
weight $4$ in $R^{(2)}$.

\emph{Example~\ref{ex5}:} we have already noted that $\O(2)$ realizes
$\P(1,1,2,2)$ as a rank $3$ quadric in $\P^4$.

\emph{Example~\ref{ex6}:} $\P(2,3,3,4)$ is the complete intersection of a
quadric and a cubic in $\P(1^5,2)$, and thus it is a degeneration of a
cubic in $\P^4$.
The algebra $R^{(6)}$ is generated by
\[
  x^3,xw,w^3,y^2,yz,z^2
\]
which have weights $6,6,12,6,6,6$ in $R$, hence $1,1,2,1,1,1$ in
$R^{(6)}$. 
This time there are two relations, namely
\[
  x^3\cdot w^3 = (xw)^3
  \quad
  \text{and}
  \quad
  y^2 \cdot z^2 =(yz)^2,
\]
which have respectively degrees $3$ and $2$ in $R^{(6)}$.
\end{rem}

We conclude with some remarks on the cases in which
Theorem~\ref{T:smoothing} does not apply because $\alpha(S) >
\alpha(S')$; the notation is still that of Table~\ref{tab:alpha}. 

\begin{prop}
Let $\P=\P(a_0,a_1,a_2,a_3)$ be a Gorenstein weighted projective
$3$-space of type \#$i$ with $i \in \{8,10,11,12,13,14\}$.
Then an anticanonical divisor of $\P$ is a double cover of the
weighted projective plane $\P(a_0,a_1,a_2)$, branched over a
bianticanonical divisor $B \in |-2K_{\P(a_0,a_1,a_2)}|$.
In all cases $2K_{\P(a_0,a_1,a_2)}$ is invertible, whereas
in all cases but \#11 the canonical sheaf $K_{\P(a_0,a_1,a_2)}$ is
not invertible.
\end{prop}

\begin{proof}
The key fact is that in all cases one has $a_3=a_0+a_1+a_2$.
Then in homogeneous coordinates $(x_0:\dots:x_3)$, a degree
$s=a_0+a_1+a_2+a_3$ homogeneous polynomial is of the form
\[
  x_3^2 + x_3\cdot f_{a_3}(x_0,x_1,x_2)
  + f_s(x_0,x_1,x_2),
\]
with $f_d$ homogeneous of degree $d$.
We may change weighted homogeneous coordinates by setting
$x_3' = x_3 + \tfrac 12 f_{a_3}$. This gives the polynomial
\[
  (x_3')^2 + f_s'(x_0,x_1,x_2)
\]
where $f_s' = f_s - f_{a_3}^2$, which defines a double cover of
$\P(a_0,a_1,a_2)$ as asserted.
The last affirmation is readily checked using the statements of
section~\ref{S:wps}.
\end{proof}

\begin{example}[\#11]
The anticanonical divisors in $\P(1,2,3,6)$ are double covers of
$\P(1,2,3)$, which is in fact a Del Pezzo surface of
degree $6$, with one $A_1$ and one $A_2$ double points, which may be
constructed by blowing up the plane $\P^2$ along three aligned
infinitely near points. 

In the embedding by $\O(6)=-\frac{1}{2}K_\P$,  $\P(1,2,3,6)$ is the cone over this
toric Del Pezzo surface, in its anticanonical embedding in $\P^6$.
It follows that $\P(1,2,3,6)\subset \bP^7$ is a limit of cones   over 
smooth Del Pezzo surfaces of degree $6$. Every such cone $T$ is
obstructed in $\mathrm{Hilb}^{\P^7}$, being in the closure of two
components, one parametrizing embedded $\bP^1\times\bP^1\times\bP^1$'s
and the other hyperplane sections of $\bP^2\times\bP^2$, as observed
in \cite[Example 4.5]{CS18}.  Therefore $[\P(1,2,3,6)]\in
\mathrm{Hilb}^{\P^7}$ is obstructed as well.
In fact the embedded versal deformation of $\P(1,2,3,6) \subset \P^7$
has been explicitly computed, see \cite{Alt} and \cite{BRS}.

On the other hand,
$\O(12)=\O(K_{\bP})$ embeds $\P(1,2,3,6)\subset \bP^{26}$ and
$\alpha(\bP, K_\P)=0$, while its general
anticanonical divisor $S$ satisfies $\alpha(S, -\restr {K_\P} S)=1$ (see Table~\ref{tab:alpha}). Therefore, by Theorem \ref{T:wahl}, $\ol{CS}$ has a unique $1$-parameter deformation to   $\bP$. 
\end{example}


\noindent
\textsc{Institut de Mathématiques de Toulouse~; UMR5219.\\
Université de Toulouse~; CNRS.\\
UPS IMT, F-31062 Toulouse Cedex 9, France.}
\\
\texttt{thomas.dedieu@math.univ-toulouse.fr}

\medskip\noindent
\textsc{Dipartimento di Matematica e Fisica \\ Università Roma Tre \\ L.go S.L. Murialdo, 1 -  00146 Roma, Italia.}\\
\texttt{sernesi@gmail.com}

\clearpage

\begin{verbatim}
Macaulay2 computations

-- open M2 with GC_INITIAL_HEAP_SIZE=100G

loadPackage "VersalDeformations"
kk = QQ -- alt: kk = ZZ/31

-- liste poids 3-dim Gwps
Wliste =
{{2, 3, 3, 4},
{1, 3, 4, 4},
{2, 3, 10, 15},
{1, 6, 14, 21},
{1, 2, 2, 5},
{1, 1, 2, 2},
{1, 2, 3, 6},
{1, 2, 6, 9},
{1, 4, 5, 10},
{1, 3, 8, 12},
{1, 1, 2, 4},
{1, 1, 1, 1},
{1, 1, 1, 3},
{1, 1, 4, 6}
}

-- pour calculer n tel que -K=O(n)
K = W -> (
    n:=0;
    for i to 3 do n=n+W_i;
    n)

-- vérification propriété N2
gVN2 = W -> (
    R:=kk[u_0..u_3, Degrees=>W];
    M:=basis(K(W),R);
    n:=numgens source M;
    S:=kk[x_0..x_(n-1)];
    T:=R**S;
    I:= ideal (for i from 0 to n-1 list (((map(T,S))(x_i)-(map(T,R))(M_(0,i)))));
    J:= eliminate (for i to 3 list (map(T,R))(u_i), I);
    p:= map (S,T, (for i to 3 list 0)|(for i to (n-1) list x_i));
    M1:= syz generators p(J);
    W,
    betti M1
    )

for i from 0 to 13 do print (gVN2(Wliste_i))

-- calcul T^1
-- wps
gV = W -> (
    R:=kk[u_0..u_3, Degrees=>W];
    M:=basis(K(W),R);
    n:=numgens source M;
    S:=kk[x_0..x_(n-1)];
    T:=R**S;
    I:= ideal (for i from 0 to n-1 list (((map(T,S))(x_i)-(map(T,R))(M_(0,i)))));
    J:= eliminate (for i to 3 list (map(T,R))(u_i), I);
    p:= map (S,T, (for i to 3 list 0)|(for i to (n-1) list x_i));
    W,
    numgens source CT^1(-1,p(J))
    )

for i from 0 to 13 do print (gV(Wliste_i))

-- K3
gS = W -> (
    R:=kk[u_0..u_3, Degrees=>W];
    M:=basis(K(W),R);
    n:=numgens source M;
    S:=kk[x_0..x_(n-1)];
    T:=R**S;
    I:= ideal (for i from 0 to n-1 list (((map(T,S))(x_i)-(map(T,R))(M_(0,i)))));
    J:= eliminate (for i to 3 list (map(T,R))(u_i), I);
    p:= map (S,T, (for i to 3 list 0)|(for i to (n-1) list x_i));
    f:=x_7+x_(n-1);
    S0:=kk[x_0..x_(n-2)];
    q:= map (S0,S, (for i to n-2 list x_i)| {0});
    J0 = q(eliminate ({x_(n-1)}, ideal (f) + p(J)));
    W,
    numgens source CT^1(-1,J0)
    )

for i from 0 to 13 do print (gS(Wliste_i))

-- can curve
gC = W -> (
    R:=kk[u_0..u_3, Degrees=>W];
    M:=basis(K(W),R);
    n:=numgens source M;
    S:=kk[x_0..x_(n-1)];
    T:=R**S;
    I:= ideal (for i from 0 to n-1 list (((map(T,S))(x_i)-(map(T,R))(M_(0,i)))));
    J:= eliminate (for i to 3 list (map(T,R))(u_i), I);
    p:= map (S,T, (for i to 3 list 0)|(for i to (n-1) list x_i));
    f0:=x_7+x_(n-1);
    f1:=x_3+x_(n-2);
    S0:=kk[x_0..x_(n-3)];
    q:= map (S0,S, (for i to n-3 list x_i)| {0, 0});
    J0 = q(eliminate ({x_(n-2),x_(n-1)}, ideal (f0,f1) + p(J)));
    W,
    numgens source CT^1(-1,J0)
    )

for i from 0 to 13 do print (gC(Wliste_i))
\end{verbatim}

\end{document}